\documentclass[b5paper]{article} 
\usepackage{amsmath, amssymb, amsthm} 
\theoremstyle{plain} 
\newtheorem{theorem}{Theorem}[section]

\newtheorem{corollary}[theorem]{Corollary}
\newtheorem{proposition}[theorem]{Proposition}
\theoremstyle{definition} 
\newtheorem{definition}[theorem]{Definition}

\renewcommand{\Re}{\operatorname{Re}}
\begin{document}

\title{Property $(T_B)$ and Property $(F_B)$ restricted to a representation without non-zero invariant vectors}

\author{Mamoru Tanaka\footnote{This research was supported by Global COE Program "Weaving Science Web beyond Particle-Matter Hierarchy", MEXT, Japan}}

\date{}

\maketitle

\begin{abstract}
In this paper, we give a necessary and sufficient condition for which a finitely generated group has a property like Kazhdan's Property $(T)$ restricted to one isometric representation on a strictly convex Banach space without non-zero invariant vectors. 
Similarly, we give a necessary and sufficient condition for which a finitely generated group has a property like Property $(FH)$ restricted to the set of the affine isometric actions whose linear part are one isometric representation on a strictly convex Banach space without non-zero invariant vectors. 
If the Banach space is the $\ell \sp p$ space ($1<p<\infty$) on a finitely generated group, 
these conditions are regarded as an estimation of the spectrum of the $p$-Laplace operator on the $\ell \sp p$ space and on the $p$-Dirichlet finite space respectively. 
\end{abstract}

\small 

\noindent 
{\bf Mathematics Subject Classification (2010).} Primary 20F65; Secondary 46B04, 47H10. \\

\noindent 
{\bf Keywords.} Finitely generated groups, isometric action, strictly convex Banach spaces.

\normalsize

\section{Introduction}

A finitely generated group $\Gamma$ is said to have Kazhdan's Property $(T)$, if every irreducible unitary representation $(\pi , H)$ does not have almost fixed point, that is, there exists a positive constant $C$ such that 
$$ \max_{\gamma \in K} \| \pi(\gamma )v-v\| \ge C\|v\| $$ 
for all $v\in H$, where $K$ is a finite generating subset of $\Gamma$.
Kazhdan's Property $(T)$ has played important roles in many different subjects (see \cite{bibBHV}). 
A finitely generated group is said to have Property $(FH)$, 
if every affine isometric action on an infinite dimensional Hilbert space has a fixed point.  
It is known that a finitely generated group has Kazhdan's Property $(T)$ if and only if it has Property $(FH)$.  

Bader, Furman, Gelander, and Monod \cite{bibBFGM} introduce a generalization of Kazhdan's Property $(T)$ and Property $(FH)$ to a Banach space $B$, and call these Property $(T_B)$ and Property $(F_B)$ respectively. 
They prove that a finitely generated group has Property $(T_{L^p([0,1])})$ for $p\in [1,\infty)$ if and only if it has Kazhdan's Property $(T)$, which is Property $(T_{L^2([0,1])})$. 
They also, and Chatterji, Dru\c{t}u and Haglund \cite{bibCDH},  prove that a finitely generated group has Property $(F_{L\sp p([0,1])})$ for $p\in [1,2]$ if and only if it has Property $(FH)$, which is Property $(F_{L^2([0,1])})$. 
On the contrary, Bourdon and Pajot \cite{bibBP} show that an infinite hyperbolic group $\Gamma $, which may have Property $(FH)$, does not have Property $(F_{L\sp p(\Gamma )})$ if $p$ is large enough. 
As this result shows, in general, Property $(FH)$ and Property $(F_B)$ are different. 

In this paper, for a strictly convex Banach space $B$, we investigate Property $(T_B)$ restricting to one linear isometric action without non-zero invariant vectors, via the variation of the displacement function with respect to the orbit of a finite generating subset of a finitely generated group. 
Also we investigate Property $(F_B)$ restricting to the set of the affine isometric actions whose linear part are one linear isometric action on $B$ without non-zero invariant vectors. 

We show the following: 
Let $\Gamma $ be a finitely generated group, $K$ a finite generating set of $\Gamma $,  and $B$ a strictly convex Banach space. 
We define the displacement function 
\begin{eqnarray*}
F_{\alpha,\ \!\! r}(v) :=\left( \sum_{\gamma\in K} \| \alpha(\gamma,v)-v\|^r m(\gamma) \right)^{1/r} \!, \ F_{\alpha,\ \!\! \infty }(v) := \max_{\gamma \in K}\| \alpha(\gamma ,v)-v\|
\end{eqnarray*} 
at $v\in B$ for an affine isometric action $\alpha$ of $\Gamma $ on $B$ and $1\le r< \infty$, where $m$ is a weight on $K$.
The absolute gradient $|\nabla _{\!\!-}F_{\alpha,\ \!\! r}|(v)$ is the maximum descent of $F_{\alpha,\ \!\! r}(v)$ around $v$ (see Definition \ref{def:ag} for details). 
Let $\pi$ be a linear isometric action  of $\Gamma $ on $B$ without non-zero invariant vectors, and $1\le r\le \infty$.

\begin{theorem}\label{thm:TB} 
The following are equivalent. 
\begin{enumerate} 
\item \label{item:thm:TB1}
There is a positive constant $C'$ such that every $v\in B$ satisfies $$\max_{\gamma \in K}\| \pi(\gamma ,v)-v\|\ge C'\|v\|.$$ 
\item \label{item:thm:TB2}
There is a positive constant $C$ such that every $v\in B\backslash\{0\}$ satisfies $$|\nabla _{\!\!-}F_{\pi,\ \!\! r}|(v) \ge C.$$ 
\end{enumerate} 
\end{theorem}

Denote by $\mathcal{A}(\pi)$ the set of the affine isometric actions whose linear part are $\pi$. 

\begin{theorem}\label{thm:FB}
The following are equivalent. 
\begin{enumerate} 
\item \label{item:thm:FB1}
Every $\alpha \in \mathcal{A}(\pi)$ has a fixed point.  
\item \label{item:thm:FB2} 
For every $\alpha \in \mathcal{A}(\pi)$, there is a positive constant $C$ such that 
every $v\in B$ with $F_{\alpha,\ \!\! r}(v) > 0$ satisfies $$|\nabla _{\!\!-}F_{\alpha,\ \!\! r}|(v) \ge C.$$ 
\end{enumerate} 
Furthermore, in {\rm (\ref{item:thm:FB2})}, $C$ can be a constant independent of each $\alpha $. 
\end{theorem}

We apply these theorems to the left regular representation $\lambda_{\Gamma ,\ \!\! p}$ of $\Gamma $ on $\ell \sp p (\Gamma )$ ($1<p<\infty$). 
Let $\Delta_p$ be the $p$-Laplace operator on $D_p(\Gamma) $ which is the Dirichlet finite function space (see Section \ref{sec:Lp} for details). 
Then we have

\begin{corollary}\label{cor:1}
The following are equivalent. 
\begin{enumerate} 
\item 
There is a positive constant $C'$ such that every $f\in \ell\sp p(\Gamma )$ satisfies $$\max_{\gamma \in K}\| \lambda_{\Gamma ,\ \!\! p}(\gamma)f-f\|_{\ell \sp p(\Gamma )}\ge C'\|f\|_{\ell \sp p(\Gamma )}.$$ 
\item 
There is a positive constant $C$ such that every $f\in \ell\sp p(\Gamma )$ satisfies  $$\|\Delta_pf\|_{\ell \sp q(\Gamma )} \ge C\| f\|_{D_p(\Gamma)}^{p-1},$$ 
where $q$ is a conjugate exponent of $p$.  
\end{enumerate} 
\end{corollary}

If $p=2$, these conditions are equivalent to a lower estimation of the spectrum of $\Delta_2$ on $\ell\sp p(\Gamma )$.

\begin{corollary}\label{cor:2}
The following are equivalent. 
\begin{enumerate}
\item Every $\alpha \in \mathcal{A}(\lambda_{\Gamma ,\ \!\! p})$ has a fixed point. 
\item There is a positive constant $C$ such that every $f\in D_p(\Gamma)$ satisfies  $$\left\| \Delta _pf \right\|_{\ell \sp q(\Gamma )} \ge C \| f\|_{D_p(\Gamma )}^{p-1} ,$$ 
where $q$ is the conjugate exponent of $p$.  
\end{enumerate}
\end{corollary}

\noindent
{\bf Acknowledgements.}
The author was supported by the Grant-in-Aid for the GCOE Program 'Weaving Science Web beyond Particle-Matter Hierarchy', Tohoku University, and GCOE 'Fostering top leaders in mathematics', Kyoto University. 
The author would like to express his gratitude to Professor Hiroyasu Izeki for helpful comments. 
The author also would like to thank the referees of this paper for their extremely valuable suggestions.


\section{Strictly convex Banach spaces}\label{sec:scBsp}

In this section, we review the definitions and several properties of strictly convex Banach  spaces, smooth Banach spaces, uniformly convex Banach spaces and uniformly smooth Banach spaces. 
Basic references are \cite{bibBL}, \cite{bibLT1} and \cite{bibLT2}. 
We denote by $(B^*,\| \ \|_{B^*})$ the dual Banach space of a Banach space $(B,\| \ \|)$. 

\begin{definition}\label{def:scBsp}
A Banach space $(B,\|\ \|)$ is said to be {\it strictly convex} if $\| v+u\|<2$ for all $v,u\in B$ with $v\not=u$, $\|v\|\le 1$ and $\| u\|\le 1$.  
\end{definition}
 
\begin{definition}\label{def:ucBsp}
A Banach space $(B,\| \ \|)$ is said to be {\it uniformly convex} if the {\it modulus of convexity} of $B$  
\begin{eqnarray*}
\delta _B (\epsilon ) := 
\inf \left\{ 1-\frac{\| u+v \|}{2} : 
\| u \| \le 1, \| v \| \le 1 \ {\rm and }\  \| u-v \| \ge \epsilon 
\right\} 
\end{eqnarray*}
is positive for all $\epsilon>0$. 
\end{definition}
A uniformly convex Banach space is strictly convex obviously. 
For instance, $L\sp p$ spaces ($1<p<\infty $) are uniformly convex Banach spaces. 

A {\it support functional} at $v\in B$ is a functional $f\in B^*$ such that $\| f\|_{B^*}=1$ and $f(v)=\| v\|$.
 
\begin{definition}\label{def:sBsp}
A Banach space is said to be {\it smooth} if every non-zero vector has a unique support functional.   
\end{definition}
 
We denote by $j(v)$ the support functional at a non-zero vector $v$ in a smooth Banach space $B$, and call $j$ the {\it duality map}.  
For the trivial vector $0$ of $B$, we set $j(0)$ to be the zero functional on $B$. 
If $B$ is a real smooth Banach space, then  
\begin{eqnarray*} 
j(v)u = \lim _{t\to 0} \frac{\| v+tu\|-\| v\|}{t}
\end{eqnarray*} 
for all $v\in B\backslash \{ 0\}$ and $u\in B$. 

\begin{definition}\label{usBsp}
A Banach space $(B,\|\ \|)$ is said to be {\it uniformly smooth}  if the {\it modulus of smoothness} of $B$ 
\begin{eqnarray*}
\rho _B (\tau ) := 
\sup \left\{ \frac{\| u+v \|}{2} + \frac{\| u-v \|}{2} -1 : 
\| u \| \le 1 \ {\rm and }\  \| v \| \le \tau \right\}
\end{eqnarray*} 
satisfies that $\rho_B (\tau )/\tau \to 0$ when $\tau\searrow 0$. 
\end{definition}

A real uniformly smooth Banach space $B$ is smooth.  
Furthermore, the duality map $j$ from the unit sphere of $B$ into the unit sphere of $B^*$ is a uniformly continuous map with a uniformly continuous inverse. 
For a complex number $c\in \mathbb{C}$, let $\Re c$ denote the real part of $c$. 
Note that for any $w^*\in B^*$ we have $\|w^*\|_{B^*}=\max\{ |\Re(w^*(v))| : v\in B, \|v\|=1\} $. 
This is because for any $w^*\in B^*$  and any $v\in B$ there is $t\in \mathbb{C}$ such that $\|t\|=1$ and 
$w^*(tv)\in \mathbb{R}$. 
The following proposition for the case that $B$ is real is Proposition A.5. in \cite{bibBL}.

\begin{proposition}\label{prop:usconti} 
Let $B$ be a uniformly smooth Banach space. 
Then 
\begin{eqnarray*}
\|j(v) -j(u)\|_{B^*} 
\le 2\rho_B
\left(2\left\| \frac{v}{\| v\|} -\frac{u}{\| u\|}\right\| \right) 
\bigg/ 
\left\| \frac{v}{\| v\|} -\frac{u}{\| u\|}\right\|
\end{eqnarray*} 
for all $v,u\in B\backslash \{0\}$ with $v\not = u$. 
\end{proposition}

\begin{proof}
For $u\in B\backslash \{0\}$ and $v\in B$, we have 
\begin{eqnarray*}
\Re (j(u)v) + \| u\| =  \Re (j(u)(v + u)) \le |j(u)(v + u)| \le \| v+u\| .
\end{eqnarray*}
Hence $\Re (j(u)v) \le \| u+v\|- \| u\| $. 

Fix $x,y\in B\backslash\{0\}$ with $x\not=y$. 
Since any $u\in B\backslash \{0\}$ satisfies $j(u)=j(u/\|u\|)$, we may assume that $\| x\| =\| y\| =1$. 
Take an arbitrary $z\in B$ with $\| z\| = \| x - y\|$. 
Then 
\begin{eqnarray*}
\Re ((j(y)- j(x))z) 
&=& \Re (j(y)z) - \Re (j(x)z) \\
&\le & \| y+z\| -\| y\| - \Re (j(x)z) + \| x\| -\Re (j(x)y) \\ 
&=& \| y+z\| - 1 + \Re (j(x)(x-y-z)) \\ 
&\le & \| y+z\| - 1 + \| x + (x-y-z)\| - \| x\| \\ 
&=& \| x + (y - x +z)\| + \| x - (y -x + z)\| -2 \\ 
&\le & 2 \rho_B (\| y-x+z\|) \\ 
&\le & 2 \rho_B (2\| y-x\|) , 
\end{eqnarray*}
because $\rho_B$ is nondecreasing and $\| y-x+z\|\le 2\| y-x\|$. 
Since $z$ is arbitrary, the proposition follows. 
\end{proof}


\section{Affine isometric actions on a strictly convex Banach space}\label{sec:IA} 

In this section, we summarize some definitions and results which relate to an isometric action $\alpha $ of a finitely generated group on a strictly convex Banach space. 
We will introduce a nonnegative continuous function $F_{\alpha ,\ \!\! r}$ on the Banach space which plays the most important role in this paper, and investigate its behavior using its absolute gradient. 

Let $\Gamma $ be a finitely generated group and $K$ a finite generating subset of $\Gamma $. 
We may assume $K$ is symmetric, that is, $K^{-1}=K$. 
We call a positive function $m$ on $K$ satisfying $ \sum_{\gamma \in K} m(\gamma ) = 1$ a {\it weight} on $K$. 
A weight $m$ on a symmetric finite generating subset $K$ is said to be {\it symmetric} if it satisfies $m(\gamma ) =m(\gamma ^{-1})$ for all $\gamma \in K$. 

Let $\pi $ be a linear isometric action of $\Gamma $ on a Banach space $B$. 
A map $c:\Gamma \to B$ is called a {\it $\pi$-cocycle} if it satisfies $c(\gamma \gamma ')=\pi (\gamma)c(\gamma ') + c(\gamma) $ for all $\gamma , \gamma '\in \Gamma $. 
A cocycle is completely determined by its values on $K$.
For an affine isometric action $\alpha $, 
there are a linear isometric action $\pi$ and a $\pi$-cocycle $c$ 
such that $\alpha (\gamma, v) = \pi (\gamma, v) + c(\gamma)$ for each $\gamma \in \Gamma $ and $v\in B$. 
We call $\pi$ the {\it linear part} of $\alpha $ and $c$ the {\it cocycle part} of $\alpha $, and we write $\alpha =\pi +c$. 
We denote by $\mathcal{A}(\pi)$ the set of the affine isometric actions whose linear part is $\pi$. 

We denote by $Z^1(\pi )$ the linear space consisting of all $\pi$-cocycles.
We define a linear map $d:B\to Z^1(\pi)$ by $dv(\gamma) := \pi (\gamma)v-v$ for each $v\in B$ and $\gamma \in \Gamma$. 
Here, for $v\in B$, we have $dv(\gamma \gamma ') = \pi(\gamma )dv(\gamma ') + dv(\gamma ) $
for all $\gamma ,\gamma '\in \Gamma $, hence $d$ is well-defined. 
We set $B^1(\pi):=d(B)$, and we call an element in $B^1(\pi)$ a {\it $\pi $-coboundary}. 
It is a linear subspace of $Z^1(\pi)$.  
If $\pi $ has no non-zero invariant vector, then $d$ is an isomorphism from $B$ onto $B^1(\pi )$. 

The space $Z^1(\pi)$ describes $\mathcal{A}(\pi)$. 
Each $\pi $-coboundary corresponds to such an affine isometric action having a fixed point. 
{\it The first cohomology of $\Gamma$ with $\pi $-coefficient} is $H^1(\Gamma,\pi ) := Z^1(\pi )/B^1(\pi )$. 
Note that $H^1(\Gamma,\pi )$ vanishes if and only if every affine isometric action $\alpha $ of $\Gamma$ on $B$ with the linear part $\pi $ has a fixed point. 

We endow $Z^1(\pi)$ with the norm 
\begin{eqnarray*}
\| c \|_r 
:= \left( \sum_{\gamma\in K} \| c(\gamma)\|^r m(\gamma) \right)^{1/r} 
\end{eqnarray*} 
for $1\le r<\infty $, or the norm 
\begin{eqnarray*}
\| c\|_{\infty } := \max_{\gamma\in K}\| c(\gamma)\|  . 
\end{eqnarray*}
Then $Z^1(\pi )$ becomes a Banach space with respect to each of these norms. 
Note that, in general, $B^1(\pi )$ is not closed in $Z^1(\pi )$. 
Since $\| dv\|_r  \le 2\| v\|$ for all $v\in B$ and $1\le r\le \infty $, 
$d$ is bounded with respect to each of these norms. 

\begin{definition}\label{def:F} 
For an affine isometric action $\alpha=\pi +c$ and $1\le  r\le \infty $, 
we define $F_{\alpha,\ \!\! r}:B\to [0,\infty)$ by 
$F_{\alpha,\ \!\! r} (v):= \|dv + c\|_r = \| \alpha (\cdot, v )-v\|_r$
for each  $v\in B$ and $1\le r\le \infty $.  
\end{definition}

The function $F_{\alpha,\ \!\! r}$ vanishes at $v_0\in B$ if and only if $v_0$ is a fixed point of $\alpha$.  
Using Minkowski's inequality, we obtain $|F_{\alpha ,\ \!\! r} (u) -F_{\alpha ,\ \!\! r} (v)| \le 2\| u-v\|$ for all $u,v\in B$, and hence $F_{\alpha ,\ \!\! r}$ is uniformly continuous for each $1\le r\le\infty $. 

A function $F$ on a strictly convex Banach space $B$ is said to be {\it convex} if, for any segment $c:[0,l]\to B$, $F(c(tl))\le (1-t)F (c(0)) + tF (c(l))$ for $t\in [0,1]$. 
For an affine isometric action $\alpha $ on a strictly convex Banach space, 
$F_{\alpha,\ \!\! r} $ is convex for each $1\le r\le\infty $ by an easy computation. 

\begin{definition}\label{def:ag}
We define the \textit{absolute gradient} $|\nabla _{\!\!-}F_{\alpha ,\ \!\! r}|$ of $F_{\alpha ,\ \!\! r}$ at $v\in B$ by 
\begin{eqnarray*} 
|\nabla _{\!\!-} F_{\alpha ,\ \!\! r}|(v) 
:= \max \left\{\limsup_{u\to v,\ \!\! u\in B} 
\frac{F_{\alpha ,\ \!\! r}(v)-F_{\alpha ,\ \!\! r}(u)}{\|v-u\|} ,\ \!  0\right\}. 
\end{eqnarray*} 
\end{definition} 

We can regard the function $|\nabla _{\!\!-}F_{\alpha,r}|$ as the size of the gradient in the direction which decreases $F_{\alpha, r}$ most.  
Note that $|\nabla _{\!\!-}F_{\alpha,\ \!\!  r}|(v) \le 2$ for any $v\in B$. 
The absolute gradient $|\nabla _{\!\!-}F_{\alpha,r}|$ has the following properties: 
The following Proposition \ref{prop:agsup}, Corollary \ref{cor:agmini} and Proposition \ref{prop:lowsemiconti} were proved by Mayer \cite{bibM} for a Hadamard space. 
His proofs are valid for Banach spaces.  

\begin{proposition}[\cite{bibM}, Proposition 2.34]\label{prop:agsup}
\begin{eqnarray*}
|\nabla _{\!\!-}F_{\alpha ,\ \!\! r}|(v) 
= \max \left\{ 
\sup_{u\not= v,\ \!\!  u\in B} \frac{F_{\alpha ,\ \!\! r}(v)-F_{\alpha ,\ \!\! r}(u)}{\|v-u\|} 
,\ \! 0\right\} 
\end{eqnarray*} 
at all $v\in B$.  
\end{proposition}

\begin{corollary}[\cite{bibM}, Corollary 2.35]\label{cor:agmini} 
A point $v_0\in B$ minimizes $F_{\alpha ,\ \!\! r}$ if and only if $|\nabla_{\!\!-}F_{\alpha ,\ \!\! r}|$ vanishes at $v_0$.
\end{corollary} 

\begin{proposition}[\cite{bibM}, Proposition 2.25]\label{prop:lowsemiconti}
The absolute gradient $|\nabla_{\!\!-}F_{\alpha ,\ \!\! r}|$ is lower semicontinuous on $B$. 
\end{proposition}


\section{A proof of Theorem \ref{thm:TB}  and Theorem \ref{thm:FB}}
In this section, we will give a proof of Theorem \ref{thm:TB}  and Theorem \ref{thm:FB}.

\begin{proof}[Proof of Theorem \ref{thm:TB}]
Note that (\ref{item:thm:TB1}) is equivalent to the condition that 
there is a positive constant $C'$ such that $F_{\pi ,r}(v)\ge C'\| v\| $ for all $v\in B$.
Note that $F_{\pi , r}$ is convex, and $F_{\pi, r}(av)= aF_{\pi, r}(v)$ for $a>0$ and $v\in B$. 
If we assume (\ref{item:thm:TB1}), we get 
\begin{eqnarray*}
|\nabla _-F_{\pi,r}|(v) 
\ge \lim_{t\to 0,\ t>0} \frac{F_{\pi,r}(v)-F_{\pi,r}(tv)}{\| v-tv\|} 
= \frac{(1-t)F_{\pi,r}(v)}{(1-t)\| v\|} 
\ge C'. 
\end{eqnarray*}
Therefore we have (\ref{item:thm:TB2}). 

On the other hand, we assume (\ref{item:thm:TB2}).  
Hence, by Proposition \ref{prop:agsup},  for all $v\in B\backslash \{0\}$
\begin{eqnarray*}
\sup_{u\in B\backslash \{v\}}
\frac{F_{\pi,r}(v)-F_{\pi,r}(u)}{\| v-u\|} \ge C. 
\end{eqnarray*} 
In particular, $F_{\pi,r}(v)>0$ for all $v\in B\backslash \{0\}$.
Besides we assume that (i) is false, that is, 
for every $\epsilon' >0$ there is a non-zero vector $v\in B\backslash \{0\}$ such that $F_{\pi,r}(v)<\epsilon' \| v\|$.  
Then for $0<\epsilon< 1$ we can take $w_0\in B $ such that $\| w_0 \|=1$ and $F_{\pi,r}(w_0)< (1-\epsilon)\epsilon C $. 
Set for $w\in B$ 
\begin{eqnarray*}
P(w):= \left\{  u\in B\backslash \{w\} : \frac{F_{\pi,r}(w)-F_{\pi,r}(u)}{\| w-u\|} \ge (1-\epsilon )C \right\} .
\end{eqnarray*}
By the assumption, $P(w)$ is not empty for any $w\in B\backslash \{0\}$. 
Since $F_{\pi,r}(0)=0$ and $F_{\pi,r}(u)\ge 0$ for any $u\in B$, $P(w)$ does not contain the origin $0$ of $B$ for any $w\in B\backslash \{0\}$. 
For $u\in P(w_0)$, we have  
\begin{eqnarray*} 
(1-\epsilon )C\| w_0-u\| \le F_{\pi,r}(w_0) < (1-\epsilon)\epsilon C
\end{eqnarray*}
and hence, $ \| w_0-u\| < \epsilon $. 
Therefore $ \| u\|>1-\epsilon $ for all $u\in P(w_0)$.

First, consider the case where $\inf _{v\in P(w_0)}F_{\pi,r}(v)\not=0$. 
Take $w_1\in  P(w_0)$ such that 
$$F_{\pi,r}(w_1) \le (1+\epsilon )\inf_{v\in P(w_0)}F_{\pi,r}(v).$$ 
Since $w_1\in P(w_0)$, for any $v\in P(w_1)$, 
we have 
\begin{eqnarray*}
\frac{F_{\pi,r}(w_0)-F_{\pi,r}(v)}{\|w_0-v\|} 
&\ge & \frac{(F_{\pi,r}(w_0)-F_{\pi,r}(w_1)) + (F_{\pi,r}(w_1)-F_{\pi,r}(v))}
{\|w_0-w_1\| + \|w_1-v\|} \\
&\ge & \frac{(1-\epsilon )C\|w_0-w_1\| 
+ (1-\epsilon )C\|w_1-v\|}
{\|w_0-w_1\| + \|w_1-v\|} \\
&=& (1-\epsilon )C. 
\end{eqnarray*} 
Hence $v \in P(w_0)$ holds, 
that is, $P(w_1) \subset P(w_0)$.  
Thus $\inf_{v\in P(w_1)}F_{\pi,r}(v)\not= 0$.  
Inductively, for each $i\in \mathbb{N}$, 
we can take $w_i \in P(w_{i-1}) $ such that 
$F_{\pi,r}(w_i) \le (1+\epsilon ^i)\inf_{v\in P(w_{i-1})}F_{\pi,r}(v)$.  
Then we have $P(w_i)\subset P(w_{i-1})$ for each $i\in \mathbb{N}$   
and $\inf_{v\in P(w_i)}F_{\pi,r}(v)\not= 0$.  
Thus for $u\in P(w_i)$ we have 
\begin{eqnarray*}
\|w_i-u\| 
&\le & \frac{F_{\pi,r}(w_i) - F_{\pi,r}(u)}{(1-\epsilon )C} \\
&\le & \frac{(1+\epsilon ^i)\inf_{v\in P(w_{i-1})}F_{\pi,r}(v) 
- \inf_{v\in P(w_i)}F_{\pi,r}(v)}{(1-\epsilon )C} \\
&\le & \frac{(1+\epsilon ^i)\inf_{v\in P(w_{i-1})}F_{\pi,r}(v) 
- \inf_{v\in P(w_{i-1})}F_{\pi,r}(v)}{(1-\epsilon )C} \\
&=& \frac{\epsilon ^i\inf_{v\in P(w_{i-1})}F_{\pi,r}(v)}
{(1-\epsilon )C} .  
\end{eqnarray*}
Since $w_i\in P(w_{i-1})$ and 
$F_{\pi,r}(w_j) \le F_{\pi,r}(w_{j-1})$ for each $j\in \mathbb{N}$, 
we have  
\begin{eqnarray*}
\|w_i-v\| 
\le \frac{\epsilon ^i F_{\pi,r}(w_i) }{(1-\epsilon )C} 
\le \frac{\epsilon ^i F_{\pi,r}(w_0) }{(1-\epsilon )C}
\end{eqnarray*}
for all $v\in P(w_i)$. 
Therefore, 
for any $\epsilon ' > 0$, there exists $i\in \mathbb{N}$ 
such that, for every $j,k \ge i$, 
\begin{eqnarray*}
\|w_j-w_k\| 
\le \operatorname{diam} P(w_i) \le 2\frac{\epsilon ^i F_{\pi,r}(w_0)}
{(1-\epsilon )C} < \epsilon '. 
\end{eqnarray*}
Since $B$ is complete, the sequence $\{ w_i\}$ converges to some point $w_{\infty }\in B$.
We have $\| w_{\infty }\| \ge 1-\epsilon $, in particular, $w_{\infty } \not= 0$,
because $ \| w_i\|>1-\epsilon $ for all $i\in \mathbb{N}$.
Since the function 
\begin{eqnarray*}
F'_i(v) := \frac{F_{\pi,r}(w_i)-F_{\pi,r}(v)}{\|w_i-v\| }
\end{eqnarray*}
is upper semicontinuous on $B \backslash \{w_i\}$, the subset 
\begin{eqnarray*}
\{v\in B \backslash \{w_i\} : 
F'_i(v) < (1-\epsilon)C \} 
&=& B \backslash 
(\{v\in B \ : 
 F'_i(v) \ge (1-\epsilon)C \} \cup \{w_i\}) \\
&=& B \backslash (P(w_i)\cup \{w_i\})
\end{eqnarray*} 
is open, that is, $P(w_i) \cup \{w_i\}$ is closed for every $i$.
Hence  $\lim_{i\to \infty}\operatorname{diam} P(w_i)=0$ implies that $\bigcap_{i=0}^{\infty } (P(w_i) \cup \{w_i\})= \{w_{\infty }\}$.
However, by the assumption there exists $v_0\in B\backslash \{w_\infty \}$ such that 
\begin{eqnarray*}
 \frac{F_{\pi,r}(w_{\infty })-F_{\pi,r}(v_0)}
{\|w_{\infty }-v_0\|} \ge (1-\epsilon )C. 
\end{eqnarray*}   
Since $w_{\infty }\in P(w_{i+1}) \cup\{w_{i+1}\} \subset P(w_i)$, 
$F_{\pi,r}(v_0)<F_{\pi,r}(w_{\infty }) < F_{\pi,r}(w_i)$ for each $i$, in particular, $w_i \not= v_0$.
Thus we have  
\begin{eqnarray*}
\frac{F_{\pi,r}(w_i)-F_{\pi,r}(v_0)}{\|w_i-v_0\|} 
&\ge & \frac{(F_{\pi,r}(w_i)-F_{\pi,r}(w_{\infty })) + (F_{\pi,r}(w_{\infty })-F_{\pi,r}(v_0))}
{\|w_i-w_{\infty }\| + \|w_{\infty }-v_0\|} \\
&\ge & \frac{(1-\epsilon )C\|w_i-w_{\infty }\| + (1-\epsilon )C\|w_{\infty }-v_0\|}
{\|w_i-w_{\infty }\| + \|w_{\infty }-v_0 \|} \\
&=& (1-\epsilon )C 
\end{eqnarray*}
for every $i$. 
This implies that 
$v_0\in \bigcap _{i=1}^{\infty } (P(w_i) \cup \{w_i\}) 
= \{w_{\infty } \}$, 
that is, $w_{\infty } = v_0$. 
This contradicts $v_0\in B\backslash \{w_{\infty }\}$. 

Secondly, we treat the case where $\inf _{v\in P(w_0)}F_{\pi,r}(v)=0$.  
Take $w_1'\in P(w_0)$ such that $F_{\pi,r}(w_1')\le  \epsilon F_{\pi,r}(w_0)$. 
As the first case, $P(w'_1)  $ is a subset of $P(w_0)$.  
If $\inf _{v\in P(w'_1)}F_{\pi,r}(v)\not =0$, then we can deduce a contradiction as the first case. 
Hence, inductively, for each $i\in \mathbb{N}$, we suppose that $\inf_{v\in P(w'_i)} F_{\pi,r}(v)=0 $. 
Take $w'_i\in P(w'_{i-1})$ such that $F(w'_i) \le \epsilon F(w'_{i-1})$. 
Then we have $P(w'_i)\subset P(w'_{i-1})$ for each $i\in \mathbb{N}$.  
Thus for $u\in P(w'_i)$ we have 
\begin{eqnarray*}
\| w_i'-u\|  
\le  \frac{F(w'_i) - F(u)}{(1-\epsilon )C} 
\le  \frac{\epsilon F(w'_{i-1})}{(1-\epsilon )C} 
\le  \frac{\epsilon ^iF(w'_0)}{(1-\epsilon )C} .  
\end{eqnarray*} 
As the first case, $w'_i$ converges to some $w'_{\infty }\in B$ with $\| w'_{\infty }\|\ge 1-\epsilon $, 
and  
$\bigcap_{i=0}^{\infty } (P(w'_i) \cup \{w' _i\})= \{w' _{\infty }\}$.
Therefore we can deduce a contradiction as the first case. 

\end{proof}

\begin{proof}[Proof of Theorem \ref{thm:FB}]
Since $F_{\alpha,\ \!\! r}$ is continuous and convex, $\inf_{v\in B}|\nabla _{\!\!-}F_{\alpha,\ \!\! r}|(v) = 0$ by Lemma 5.4 in \cite{bibT}. 
Hence, if condition (\ref{item:thm:FB2}) holds, there exists $x_0\in N$ with $F_{\alpha,\ \!\! r}(x_0) = 0$. 
The point $x_0$ is a fixed point of $\alpha$. 

Suppose (i). 
Condition (i) is equivalent to the condition that the first cohomology $H^1(\Gamma ,\pi )$ vanishes, that is, $B^1(\pi )$ coincides with $Z^1(\pi )$. 
Since $\pi$ does not have non-zero invariant vectors, $d:B\to B^1(\pi )$ is one-to-one. 
Hence the open mapping theorem implies that the inverse map $d^{-1}$ of $d$ is bounded. 
Thus there exists $C>0$ satisfying $\| v\| = \| d^{-1} (dv) \| \le C \| dv\|_r $ 
for all $v\in B$. 
Take an arbitrary affine isometric action $\alpha \in \mathcal{A}(\pi)$. 
Then there exists a fixed point $v_0\in B$ of $\alpha $. 
Since $\pi (\gamma )v = \alpha (\gamma ) (v+v_0) - v_0$ for all $v\in B$ and $\gamma \in \Gamma $, we have $F_{\alpha ,\ \!\! r}(v+v_0)= F_{\pi ,\ \!\! r}(v)$ for all $v\in B$. 
Therefore we may assume that $\alpha $ coincides with $\pi$. 
By the definition of $d$, we have $F_{\pi ,\ \!\! r} (v) = \| dv\|_r$ for all $v\in B$. 
Hence we have 
$$ | \nabla _{\!\!-}  F_{\pi,\ \!\! r} |(v) 
\ge \lim_{\epsilon \to 0} 
\frac{ F_{\pi,\ \!\! r}(v)- F_{\pi,\ \!\! r}(v_\epsilon )} 
{\| v-v_\epsilon \|} 
=\lim_{\epsilon \to 0} 
\frac{\| dv \|_r - \| dv_\epsilon  \|_r} 
{ \epsilon \| v\|} 
= \lim_{\epsilon \to 0} 
\frac{\epsilon \| dv \|_r} 
{ \epsilon \| v\|} 
\ge  \frac{1}{C}  $$ 
for all non-zero vectors $v\in B$, where $v_\epsilon :=(1-\epsilon )v$ for $\epsilon >0$. 
Since $F_{\pi ,\ \!\! r}(0)=0$, we have completed the proof. 
Since $C$ is independent of each $\pi$-cocycle, the constant $C'$ in the theorem can be independent of each $\alpha $.  
\end{proof}


\section{A description of the absolute gradient of $F_{\alpha ,\ \!\! p}$}

In this section, we see a description of the absolute gradient of $F_{\alpha ,\ \!\! p}$ for an affine isometric action $\alpha $ of a finitely generated group $\Gamma $ on some Banach space $B$ and $1<p<\infty$.  
Suppose that the finite generating set $K$ and the weight $m$ is symmetric in this section. 

\begin{proposition}\label{prop:agscsBsp}
Suppose that 
$B$ is strictly convex, smooth and real.  
Let $\alpha = \pi +c$ be an affine isometric action of $\Gamma $ on $B$. 
Then for $v\in B$ with $F_{\alpha,\ \!\! p}(v)>0$ we have 
\begin{eqnarray*} 
&& |\nabla _{\!\!-}F_{\alpha,\ \!\! p}|(v) \\ 
&=& \frac{1}{F_{\alpha,\ \!\! p}(v)^{p-1}}
\sup_{v\in B;\|u\|=1}\sum_{\gamma\in K} \|\alpha (\gamma)v-v \|^{p-1}
j(\alpha (\gamma)v -v) (\pi(\gamma)u-u)  m(\gamma)\\
&=& 
\frac{2}{F_{\alpha,\ \!\! p}(v)^{p-1}}
\left\| \sum_{\gamma\in K} 
\|\alpha (\gamma)v-v\|^{p-1} m(\gamma) j(\alpha (\gamma)v-v) 
\right\|_{B^*} . 
\end{eqnarray*}
\end{proposition}

\begin{proof} 
Fix $v\in B$ such that $F_{\alpha,\ \!\! p}(v)>0$. 
Since $F_{\alpha ,\ \!\! p}$ is convex, for $t\ge s >0$, we have 
\begin{eqnarray*}
F_{\alpha,\ \!\! p} (v+su) \le \left( 1-\frac{s}{t}\right) F_{\alpha,\ \!\! p}(v) + \frac{s}{t} F_{\alpha,\ \!\! p}(v+tu). 
\end{eqnarray*}
Therefore we have
\begin{eqnarray*}
\frac{F_{\alpha,\ \!\! p}(v)-F_{\alpha,\ \!\! p}(v+tu)}{t} \le \frac{F_{\alpha,\ \!\! p}(v)-F_{\alpha,\ \!\! p}(v+su)}{s} . 
\end{eqnarray*}
This implies that 
\begin{eqnarray*}
\lim_{\epsilon \to 0} \frac{F_{\alpha,\ \!\! p}(v)-F_{\alpha,\ \!\! p}(v+\epsilon u)}{\epsilon }
= \sup_{s >0} \frac{F_{\alpha,\ \!\! p}(v)-F_{\alpha,\ \!\! p}(v+su)}{s}. 
\end{eqnarray*} 
Hence we have 
\begin{eqnarray*}
\limsup_{u\to v,\ \!\! u\in B}
\frac{F_{\alpha,\ \!\! p}(v)-F_{\alpha,\ \!\! p}(u)}{\|v-u\|} 
= \sup_{u\in B;\| u\|=1}
\lim_{\epsilon \to 0} 
\frac{F_{\alpha,\ \!\! p}(v)-F_{\alpha,\ \!\! p}(v +\epsilon u)}
{\epsilon } . 
\end{eqnarray*}
To calculate the right hand side, we use an inequality in \cite[(2.15.1)]{bibHLP}: 
\begin{eqnarray*}
pb^{p-1}(a-b) \le a^p-b^p \le pa^{p-1}(a-b) 
\end{eqnarray*} 
for $a,b>0$. 
Set $Du(\gamma):= \alpha (\gamma)u-u$ for each $\gamma \in K$ and $u\in B$. 
Then, for a small $\epsilon >0$, we have 
\begin{eqnarray*}
&&\frac{F_{\alpha,\ \!\! p}(v)-F_{\alpha,\ \!\! p}(v +\epsilon u)}{\epsilon } \\
&\le & \frac{F_{\alpha,\ \!\! p}(v)^p-F_{\alpha,\ \!\! p}(v +\epsilon u)^p}
{pF_{\alpha,\ \!\! p}(v +\epsilon u)^{p-1}\epsilon } \\
&=& \sum_{\gamma\in K} \frac{ \|Dv(\gamma )\|^p  
- \|D(v+\epsilon u)(\gamma )\| ^p }
{pF_{\alpha,\ \!\! p}(v +\epsilon u)^{p-1}\epsilon } m(\gamma)\\ 
&\le &  \sum_{\gamma\in K} 
\left( 
\frac{\|Dv(\gamma)\|^{p-1}}
{F_{\alpha,\ \!\! p}(v +\epsilon u)^{p-1}} 
\frac{\|Dv(\gamma)\| - \|D(v+\epsilon u)(\gamma)\|}{\epsilon }
\right) m(\gamma) . 
\end{eqnarray*}
Similarly, we have 
\begin{eqnarray*}
&&\frac{F_{\alpha,\ \!\! p}(v)-F_{\alpha,\ \!\! p}(v +\epsilon u)}{\epsilon } \\ 
&\ge & \sum_{\gamma\in K} \left( 
\frac{\|D(v+\epsilon u)(\gamma) \|^{p-1}}
{F_{\alpha,\ \!\! p}(v)^{p-1}} 
\frac{\|Dv(\gamma)\| - \|D(v+\epsilon u)(\gamma)\|}
{\epsilon }\right) m(\gamma). 
\end{eqnarray*}
Since $B$ is real and smooth, for $\gamma \in K$ such that $Dv(\gamma ) \not= 0$, 
\begin{eqnarray*}
\lim _{\epsilon \to 0} 
\frac{\| Dv(\gamma)\| - \|D(v+\epsilon u)(\gamma )\| }
{\epsilon}  
&=& \lim _{\epsilon \to 0} 
\frac{\| Dv(\gamma )\| - \|Dv(\gamma ) + \epsilon du(\gamma )\|}
{\epsilon} \\  
&=& -j(Dv(\gamma )) (du(\gamma )) , 
\end{eqnarray*}
and, for $\gamma \in K$ such that $Dv(\gamma) =0$,  
\begin{eqnarray*}
\lim _{\epsilon \to 0} 
\frac{\| Dv(\gamma )\| - \|D(v+\epsilon u)(\gamma )\| }
{\epsilon}  
= \lim _{\epsilon \to 0} 
\frac{-\|\epsilon du(\gamma )\|}
{\epsilon} 
= -\|du(\gamma )\| . 
\end{eqnarray*}
Therefore we have 
\begin{eqnarray*}
&&\lim_{\epsilon \to 0} 
\frac{F_{\alpha,\ \!\! p}(v)-F_{\alpha,\ \!\! p}(v +\epsilon u)}
{\epsilon } \\ 
&=& 
\sum_{\gamma\in K} 
\frac{\|Dv(\gamma ) \|^{p-1}}{F_{\alpha,\ \!\! p}(v)^{p-1}} 
(-j(Dv(\gamma )) du(\gamma ))  m(\gamma)\\ 
&=& \frac{1}{F_{\alpha,\ \!\! p}(v)^{p-1}} 
\sum_{\gamma\in K}  
\|Dv(\gamma)\|^{p-1} (j(Dv(\gamma))u - j(Dv(\gamma )) \pi (\gamma )u) m(\gamma) .
\end{eqnarray*} 
Since there is $u\in B$ at which this limit is nonnegative, the first equality in the proposition is proved. 
To prove the last line of this equality, we continue the computation. 
For arbitrary $\gamma \in K$, since $\pi (\gamma )$ is a surjective linear isometry,
$$\| \pi^{\#} (\gamma ^{-1}) j(Dv(\gamma )) \|_{B^*} = \| j(Dv(\gamma )) \|_{B^*}= 1$$ 
where $\pi^{\#} (\gamma ^{-1}) w^* (v) := w^* (\pi(\gamma)v) $ for $w^*\in B^*$ and $v\in B$, and 
$$\left( \pi^{\#} (\gamma ^{-1}) j(Dv(\gamma )) \right)
\left( \pi(\gamma ^{-1})Dv(\gamma ) \right)
= \| Dv(\gamma ) \|
= \| \pi(\gamma ^{-1}) Dv(\gamma )\|. $$ 
Due to the smoothness of $B$, 
$\pi^{\#} (\gamma ^{-1}) j(Dv(\gamma ))$ coincides with $j(\pi(\gamma ^{-1})Dv(\gamma )) $.  
Since $c(e)=0$ for the identity element $e$ of $\Gamma $,  
we have $\pi(\gamma ^{-1})c(\gamma ) +c(\gamma ^{-1}) = c(\gamma ^{-1}\gamma ) =0$ for all $\gamma , \gamma ' \in \Gamma $. 
Hence 
\begin{eqnarray*}
\pi(\gamma ^{-1})Dv(\gamma ) 
&=& \pi(\gamma ^{-1})\alpha (\gamma ) v - \pi(\gamma ^{-1}) v\\ 
&=& \pi(\gamma ^{-1})\pi(\gamma ) v  
+ \pi(\gamma ^{-1}) c(\gamma ) 
- \pi(\gamma ^{-1}) v\\
&=& v - c(\gamma ^{-1}) -\pi(\gamma ^{-1}) v\\
&=& - Dv(\gamma ^{-1}) . 
\end{eqnarray*}
We get $\pi^{\#} (\gamma ^{-1}) j(Dv(\gamma ))= j(-Dv(\gamma ^{-1})) = -j(Dv(\gamma ^{-1}))$. 
Because  
\begin{eqnarray*}
\| Dv(\gamma ^{-1})\| 
= \|\pi(\gamma ^{-1}) Dv(\gamma )\| 
= \| Dv(\gamma )\|  
\end{eqnarray*}
and $m$ is symmetric, we have 
\begin{eqnarray*}
&&\sum_{\gamma \in K} \|Dv(\gamma )\|^{p-1}
(j(Dv(\gamma))u - j(Dv(\gamma )) (\pi (\gamma )u)) m(\gamma)\\
&=& \sum_{\gamma \in K}
\|Dv(\gamma )\|^{p-1} (j(Dv(\gamma ))u + j(Dv(\gamma ^{-1}))u) m(\gamma)\\
&=& 2\sum_{\gamma \in K}
\|Dv(\gamma ) \|^{p-1} (j(Dv(\gamma ) ) u) m(\gamma). 
\end{eqnarray*}
Therefore we obtain 
\begin{eqnarray*}
&&\limsup_{u\to v,\ \!\! u\in B}
\frac{F_{\alpha,\ \!\! p}(v)-F_{\alpha,\ \!\! p}(u)}{\|v-u\|} \\
&=& \sup_{u\in B;\| u\|=1} \frac{1}{F_{\alpha,\ \!\! p}(v)^{p-1}}
\left(2 \sum_{\gamma \in K}
\|Dv(\gamma) \|^{p-1} m(\gamma)j(Dv(\gamma )) \right)u \\
&=& \frac{2}{F_{\alpha,\ \!\! p}(v)^{p-1}} 
\left\| \sum_{\gamma \in K}
\|Dv(\gamma)\|^{p-1}m(\gamma)j(Dv(\gamma)) \right\|_{B^*} .  
\end{eqnarray*}   
\end{proof}

\begin{proposition}\label{prop:agucusBsp}
Suppose that $B$ is uniformly convex and uniformly smooth. 
Let $\alpha = \pi +c$ be an affine isometric action of $\Gamma $ on $B$.
Then for $v\in B$ with $F_{\alpha,\ \!\! p}(v)>0$ we have 
\begin{eqnarray*} 
&&|\nabla _{\!\!-}F_{\alpha,\ \!\! p}|(v) \\ 
&=& \frac{1}{F_{\alpha,\ \!\! p}(v)^{p-1}}
\sup_{v\in B;\|u\|=1}\sum_{\gamma\in K} \|\alpha (\gamma)v -v\|^{p-1}
\Re j(\alpha (\gamma)v-v) (\pi (\gamma)u-u)  m(\gamma)\\
&=& \frac{2}{F_{\alpha,\ \!\! p}(v)^{p-1}}
\left\| \sum_{\gamma\in K} 
\|\alpha (\gamma)v-v\|^{p-1} m(\gamma) j(\alpha (\gamma)v-v) 
\right\|_{B^*} . 
\end{eqnarray*} 
\end{proposition}

\begin{proof} 
Set $Du(\gamma):= \alpha (\gamma)u-u$ for each $u\in B$ and $\gamma \in K$.  
Fix $v\in B$ such that $F_{\alpha,\ \!\! p}(v)>0$. 
As in the proof of Proposition \ref{prop:agscsBsp}, for a small $\epsilon >0$ we have 
\begin{eqnarray*}
&&\sum_{\gamma\in K} 
\left( 
\frac{\|Dv(\gamma)\|^{p-1}}
{F_{\alpha,\ \!\! p}(v +\epsilon u)^{p-1}} 
\frac{\|Dv(\gamma)\| - \|D(v+\epsilon u)(\gamma)\|}{\epsilon }
\right) m(\gamma)\\ 
&\ge & \frac{F_{\alpha,\ \!\! p}(v)-F_{\alpha,\ \!\! p}(v +\epsilon u)}{\epsilon }\\  
&\ge & \sum_{\gamma\in K} \left( 
\frac{\|D(v+\epsilon u)(\gamma) \|^{p-1}}
{F_{\alpha,\ \!\! p}(v)^{p-1}} 
\frac{\|Dv(\gamma)\| - \|D(v+\epsilon u)(\gamma)\|}{\epsilon }
\right) m(\gamma). 
\end{eqnarray*}
Because $D(v+\epsilon u)(\gamma ) = Dv(\gamma ) + \epsilon du(\gamma )$, for $\gamma \in K$ such that $Dv(\gamma ) \not =0$, we get 
\begin{eqnarray*}
\|Dv(\gamma ) \| -\| D(v+\epsilon u)(\gamma )\| 
&\le & \Re j(Dv(\gamma )) 
(Dv(\gamma ) - D(v+\epsilon u)(\gamma ) ) \\
&=& -\epsilon \Re j( Dv(\gamma )) (du(\gamma )) , 
\end{eqnarray*}
and 
\begin{eqnarray*}
\|Dv(\gamma ) \| -\| D(v+\epsilon u)(\gamma )\| 
&\ge & \Re j(D(v+\epsilon u)(\gamma ) ) 
(Dv(\gamma ) - D(v+\epsilon u)(\gamma ) ) \\
&=& -\epsilon \Re j(D(v+\epsilon u)(\gamma )) (du(\gamma )) . 
\end{eqnarray*}
Since $D(v+\epsilon u)(\gamma ) \not =0$ for small $\epsilon>0$, we obtain 
\begin{eqnarray*}
&&\left\| 
\frac{D(v+\epsilon u)(\gamma ) }{\| D(v+\epsilon u)(\gamma ) \|} 
-\frac{Dv(\gamma )}{\| Dv(\gamma )\|}\right\| \\
&=& \left\| 
\frac{Dv(\gamma ) + \epsilon du(\gamma )}
{\| D(v+\epsilon u)(\gamma ) \|}
-\frac{Dv(\gamma )}{\| D(v+\epsilon u)(\gamma ) \|}
\frac{\| D(v+\epsilon u)(\gamma ) \|}{\| Dv(\gamma )\|}\right\| \\
&=& \left\| \left( 1 
- \frac{\| D(v+\epsilon u)(\gamma ) \|}{\| Dv(\gamma )\|}
\right) \frac{Dv(\gamma ) }
{\| D(v+\epsilon u)(\gamma ) \|}
+ \epsilon \frac{du(\gamma )}
{\| D(v+\epsilon u)(\gamma ) \|}
\right\| \\
&\le & \left| 1 - \frac{\| D(v+\epsilon u)(\gamma ) \|}{\| Dv(\gamma )\|} \right| 
\frac{\|Dv(\gamma )\| }
{\| D(v+\epsilon u)(\gamma ) \|}
+ \epsilon \frac{\|du(\gamma )\|}
{\| D(v+\epsilon u)(\gamma ) \|} \\
&\to& 0 \ \ (\text{as } \epsilon \to 0). 
\end{eqnarray*} 
Hence, by Proposition \ref{prop:usconti} and the uniform smoothness of $B$, $j(D(v+\epsilon u)(\gamma ) )$ converges to $j(Dv(\gamma ))$ in $B^*$ as $\epsilon \to 0$. 
On the other hand, for $\gamma \in K$ such that $Dv(\gamma ) =0$, we have 
\begin{eqnarray*}
\frac{\| Dv(\gamma )\| - \|D(v+\epsilon u)(\gamma )\| }
{\epsilon}  
= \frac{-\|\epsilon du (\gamma )\|}
{\epsilon} 
= -\|du (\gamma )\| . 
\end{eqnarray*}
Hence we have 
$$\lim_{\epsilon \to 0} 
\frac{F_{\alpha,\ \!\! p}(v)-F_{\alpha,\ \!\! p}(v +\epsilon u)}
{\epsilon } 
= -\sum_{\gamma\in K} \left( 
\frac{\|Dv(\gamma)\|^{p-1}}{F_{\alpha,\ \!\! p}(v)^{p-1}} 
\Re j(Dv(\gamma )) (du(\gamma )) 
\right) m(\gamma) . $$
Therefore, using the equality $\|w^*\|_{B^*}=\max\{ |\Re(w^*(v))| : v\in B, \|v\|=1\} $ for $w^*\in B^*$, 
as in the proof of Proposition \ref{prop:agscsBsp}, the proposition follows.  
\end{proof}

\begin{corollary}\label{cor:agLpsp}
Let $\alpha $ be an affine isometric action of $\Gamma $ on $L\sp p(W,\nu )$, where $1<p<\infty$ and $(W,\nu )$ is a measure space. 
For any $f\in L\sp p(W,\nu )$ such that $F_{\alpha,\ \!\! p}(f)>0$, we have $$|\nabla _{\!\!-}F_{\alpha,\ \!\! p}|(f) = 2\|G_{\alpha,\ \!\! p}(f)\|_{L\sp q(W,\nu)} /F_{\alpha,\ \!\! p}(f)^{p-1} .$$ 
Here $q$ is the conjugate exponent of $p$, that is, $q=p/(p-1)$, and 
\begin{eqnarray*} 
G_{\alpha,\ \!\! p}(f)(x) = \sum_{\gamma\in K} 
| \alpha(\gamma)f(x)-f(x)|^{p-2} 
(\alpha (\gamma)f(x)-f(x))m(\gamma)  
\end{eqnarray*}
for $x\in W$, 
where $| \alpha(\gamma)f(x)-f(x)|^{p-2}=0$ if $f(x) = \alpha(\gamma)f(x)$ and $p<2$. 
\end{corollary}

\begin{proof}
For $f\in L\sp p (W,\nu )$, we have $j(f) = |f|^{p-2} \bar f/\|f\|_{L\sp p(W,\nu)}^{p-1}$, where $\bar f$ is the complex conjugation of $f$. 
Indeed, we have  
\begin{eqnarray*}
\int _W
\left(\frac{|f(x)|^{p-2} \bar f(x)}{\|f\|_{L\sp p(W,\nu )}^{p-1}} \right) 
f(x) d\nu (x)
= \int _W\frac{|f(x)|^p}{\|f\|_{L\sp p(W,\nu )}^{p-1}}  d\nu (x)
= \|f\|_{L\sp p(W,\nu )} 
\end{eqnarray*}
and  
$$
\int _W\left| 
\frac{|f(x)|^{p-2} \bar f(x)}{\|f\|_{L\sp p(W,\nu )}^{p-1}} 
\right| ^q d\nu (x) 
= \int _W\frac{|f(x)|^{(p-1)q} }{\|f\|_{L\sp p(W,\nu )}^{(p-1)q}}  d\nu (x) 
= \int _W\frac{|f(x)|^p }{\|f\|_{L\sp p(W,\nu )}^p} d\nu (x) 
= 1 
$$
We have thus proved the corollary.  
\end{proof}

Since $\alpha (\gamma )v-v=(dv+c)(\gamma )$ for all $\gamma \in \Gamma $ and  $v\in B$,  
by Proposition \ref{prop:agscsBsp} and Proposition \ref{prop:agucusBsp}, if $B$ is strictly convex, smooth and real, or uniformly convex and uniformly smooth, then for  $1<p<\infty$ 
\begin{eqnarray*} 
|\nabla _{\!\!-}F_{\alpha,\ \!\! p}|(v) 
= \frac{2}{\| dv + c \|_p^{p-1}} 
\left\| \sum_{\gamma\in K} 
\|(dv+c)(\gamma)\|^{p-1} m(\gamma) j((dv+c)(\gamma)) 
\right\|_{B^*} 
\end{eqnarray*}
for all $v\in B$ such that $\| dv + c \|_p>0$. 
Hence for $C>0$, $|\nabla _{\!\!-}F_{\alpha,\ \!\! p}|(v) \ge C$ for all $v\in B$ such that $F_{\alpha,\ \!\! p}(v)>0$  if and only if 
\begin{eqnarray*}
\left\| \sum_{\gamma\in K} 
\|(dv+c)(\gamma)\|^{p-1} m(\gamma) j((dv+c)(\gamma)) 
\right\|_{B^*} 
\ge \frac{C}{2}\| dv + c \|_p^{p-1} 
\end{eqnarray*}
for all $v\in B$. 
From Theorem \ref{thm:TB}, we have 

\begin{corollary}\label{cor:TB} 
Let $\pi$ be a linear isometric action of $\Gamma$ on $B$ without non-zero invariant vectors. 
Suppose that $B$ is either strictly convex, smooth and real, or uniformly convex and uniformly smooth. 
Then the following are equivalent. 
\begin{enumerate} 
\item 
There is a positive constant $C'$ such that every $v\in B$ satisfies $$\max_{\gamma \in K}\| \pi(\gamma ,v)-v\|\ge C'\|v\|.$$ 
\item 
There is a positive constant $C$ such that 
\begin{eqnarray*}
\left\| \sum_{\gamma\in K} 
\|dv(\gamma)\|^{p-1} m(\gamma) j(dv(\gamma)) 
\right\|_{B^*}\ge C \| dv\|_p^{p-1} 
\end{eqnarray*} 
for all $v\in B$. 
\end{enumerate} 
\end{corollary}

There exists a one-to-one correspondence between $Z^1(\pi )$ and $\mathcal{A}(\pi)$ if $\pi $ has no non-zero invariant vector and the origin of $B$ is fixed. 
Since $dv+c$ is a $\pi$-cocycle, from Theorem \ref{thm:FB}, we have   
 
\begin{corollary}\label{cor:coho}
Let $\pi$ be a linear isometric action of $\Gamma$ on $B$ without non-zero invariant vectors. 
Suppose that $B$ is either strictly convex, smooth and real, or uniformly convex and uniformly smooth. 
Then every $\alpha \in \mathcal{A}(\pi)$ has a fixed point if and only if there exists $C>0$ such that 
\begin{eqnarray*}
\left\| \sum_{\gamma\in K} 
\|c(\gamma)\|^{p-1} m(\gamma) j(c(\gamma)) 
\right\|_{B^*}\ge C \| c\|_p^{p-1} 
\end{eqnarray*} 
for all $c\in Z^1(\pi )$. 
\end{corollary}


\section{An application of the theorems to an $\ell \sp p$ space}\label{sec:Lp} 
Let $\Gamma$ be a finitely generated infinite group, $K$ a symmetric finite generating subset of $\Gamma$, $m$ a symmetric weight on $K$, and $1<p<\infty $. 

We denote by $\mathcal{F}(\Gamma)$ the space of all complex-valued functions on $\Gamma$. 
The following argument is also valid for real-valued case. 
The {\it left regular representation} $\lambda_\Gamma$ of $\Gamma$ on $\mathcal{F} (\Gamma)$ is defined by $\lambda _\Gamma (\gamma )f(\gamma ')= f(\gamma ^{-1}\gamma ')$ for each $f\in \mathcal{F}(\Gamma)$ and each $\gamma ,\gamma '\in \Gamma $. 
We define a linear map $d$ on $\mathcal{F}(\Gamma )$ by $df(\gamma ):= \lambda_{\Gamma} (\gamma )f-f $ for each $f\in \mathcal{F}(\Gamma )$ and $\gamma \in \Gamma$. 
The Lebesgue space $\ell \sp p(\Gamma)$ is the Banach space $\{ f\in \mathcal{F}(\Gamma ): \sum_{\gamma \in \Gamma }|f(\gamma )|^p <\infty \} $ with the norm $\| f\|_{\ell \sp p(\Gamma )} := (\sum_{\gamma \in \Gamma }|f(\gamma )|^p)^{1/p}$. 
The restriction of $\lambda_\Gamma$ to $\ell \sp p(\Gamma)$ is a linear isometric action without non-zero invariant vectors, and we denote it by $\lambda_{\Gamma , p}$. 

We say that $f\in \mathcal{F}(\Gamma )$ is {\it $p$-Dirichlet finite} if $df(\gamma )\in \ell \sp p(\Gamma)$ for each $\gamma \in K$, and we denote by $D_p(\Gamma)$ the space of all $p$-Dirichlet finite functions. 
The space $\ell \sp p (\Gamma)$ is a subspace of $D_p(\Gamma)$. 
The space of all constant functions on $\Gamma $ is also a subspace of $D_p(\Gamma )$, and is regarded as $\mathbb{C}$. 
Since this is the kernel of $d$, we can define a norm on $D_p(\Gamma)/\mathbb{C}$ by 
$$\| f\|_{D_p(\Gamma )} = \left( \sum_{\gamma \in K} \| df(\gamma )\|_{\ell \sp p(\Gamma)}^p m(\gamma) \right)^{1/p} .$$ 
Since 
\begin{eqnarray*}
\lambda_{\Gamma , p}(\gamma ) df(\gamma ') + df(\gamma ) 
&=& \lambda_{\Gamma}(\gamma )\lambda_{\Gamma}(\gamma ')f - \lambda_{\Gamma}(\gamma )f  
+ \lambda_{\Gamma}(\gamma )f -f \\ 
&=& \lambda_{\Gamma}(\gamma \gamma ')f - f\\ 
&=& df(\gamma \gamma ') 
\end{eqnarray*}
for all $f\in D_p(\Gamma )$ and $\gamma ,\gamma '\in \Gamma $, we obtain $df\in Z^1(\lambda_{\Gamma , p})$ for $f\in D_p(\Gamma)$. 

Furthermore, it is proven by Puls in \cite{bibPul03} and \cite{bibPul06} that $d(D_p(\Gamma ))= Z^1(\lambda_{\Gamma , p})$. 
Recall that $B^1(\lambda_{\Gamma , p})=d(\ell \sp p(\Gamma))$. 
Therefore $d$ induces an isometric isomorphism from $D_p(\Gamma )/\mathbb{C}$ onto $ Z^1(\lambda_{\Gamma ,\ \!\! p})$ and a linear 
isomorphism from $D_p(\Gamma)/(\ell \sp p(\Gamma) \oplus \mathbb{C})$ onto $H\sp 1(\Gamma,\lambda _\Gamma)$. 
Hence, for any affine isometric action $\alpha $ on $\ell \sp p(\Gamma)$ with the linear part $\lambda_{\Gamma ,\ \!\! p}$, there exists a unique $f_{\alpha }\in D_p(\Gamma)$ up to constant functions such that the cocycle part $c$ of $\alpha $ coincides with $df_{\alpha}$ and $\| c\|_p=\| f_{\alpha } \|_{D_p(\Gamma )}$. 
In particular, $f_{\lambda_{\Gamma ,\ \!\! p}}\equiv 0$. 

The $p$-Laplacian $\Delta _pf$ of $f\in D_p(\Gamma)$ is defined by 
\begin{eqnarray*}
\Delta _pf(x) 
:= \sum_{\gamma \in K} 
|df(\gamma )(x)|^{p-2}(df(\gamma )(x))m(\gamma )  
\end{eqnarray*}
where for $p<2$ we set $|df(\gamma )(x)|^{p-2}=0$ whenever $|df(\gamma )(x)|=0$. 
Since $$F_{\alpha ,\ \!\! p}(f) = \| df + df_{\alpha }\|_p = \| f+f_{\alpha}\|_{D_p(\Gamma )}$$ for all $f\in \ell \sp p(\Gamma )$, using Corollary \ref{cor:agLpsp}, we have 
\begin{eqnarray*}
|\nabla _{\!\!-}F_{\alpha,\ \!\! p}|(f) 
=\frac{2\|\Delta_p(f+f_{\alpha})\|_{\ell \sp q(\Gamma )}} 
{\| f+f_{\alpha}\|_{D_p(\Gamma)}^{p-1}} 
\end{eqnarray*} 
for all $f\in \ell \sp p(\Gamma )$ such that $F_{\alpha ,\ \!\! p}(f)>0$. 
In particular, 
\begin{eqnarray*}
|\nabla _{\!\!-}F_{\lambda_{\Gamma , p},\ \!\! p}|(f) 
=\frac{2\|\Delta_pf\|_{\ell \sp q(\Gamma )}} 
{\| f\|_{D_p(\Gamma)}^{p-1}} 
\end{eqnarray*} 
for all $f\in \ell \sp p(\Gamma )$ such that $F_{\lambda_{\Gamma , p} ,\ \!\! p}(f)>0$. 
Hence Theorem \ref{thm:TB} implies  

\begin{corollary}\label{cor:3}
The following are equivalent. 
\begin{enumerate} 
\item 
There is a positive constant $C'$ such that every $f\in \ell\sp p(\Gamma )$ satisfies $$\max_{\gamma \in K}\| \lambda_{\Gamma ,\ \!\! p}(\gamma)f-f\|_{\ell \sp p(\Gamma )}\ge C'\|f\|_{\ell \sp p(\Gamma )}.$$ 
\item \label{item:IntroBAGP} 
There is a positive constant $C$ such that every $f\in \ell\sp p(\Gamma )$ satisfies  $$\|\Delta_pf\|_{\ell \sp q(\Gamma )} \ge C\| f\|_{D_p(\Gamma)}^{p-1},$$ 
where $q$ is a conjugate exponent of $p$.  
\end{enumerate} 
\end{corollary}
By the proof of Theorem \ref{thm:TB}, if $C''>0$ satisfies $C''\le C$, then $C''$ satisfies condition $(${\rm ii}$)$ as $C$. 
For $g\in \ell\sp q(\Gamma )$ and $f\in \ell\sp p(\Gamma )$, set $\langle g,f\rangle :=\sum_{\gamma \in \Gamma }g(\gamma )f(\gamma )$. 
Assume that there is a positive constant $C$ such that every $f\in \ell\sp p(\Gamma )$ satisfies $\langle \Delta_pf,f\rangle\ge C\| f\|^p_{\ell\sp p(\Gamma )}$. 
Then, using H${\rm \ddot{o}}$lder's inequality, we easily deduce condition (i) in Corollary \ref{cor:3}. 
On the other hand, for $f\in \ell \sp p(\Gamma )$
$$\| f\|_{D_p(\Gamma)}
=F_{\lambda_{\Gamma , p},\ \!\! p}(f)
\ge \max_{\gamma \in K}\| \lambda_{\Gamma ,\ \!\! p}(\gamma)f-f\|_{\ell \sp p(\Gamma )}/|K|^{1/p}. $$ 
Therefore, if condition (i) and condition (ii) in Corollary \ref{cor:3} holds, 
then there is a positive constant $C''$ such that every $f\in \ell\sp p(\Gamma )$ satisfies  $\|\Delta_pf\|_{\ell \sp q(\Gamma )} \ge C''\| f\|_{\ell \sp p(\Gamma)}^{p-1}$.  
In particular, if $p=2$, these represent a lower estimation of the spectrum of $\Delta_2$. 
 
Theorem \ref{thm:FB} implies 
\begin{corollary}\label{cor:4}
The following are equivalent. 
\begin{enumerate}
\item Every $\alpha \in \mathcal{A}(\lambda_{\Gamma ,\ \!\! p})$ has a fixed point. 
\item There is a positive constant $C$ such that every $f\in D_p(\Gamma)$ satisfies  $$\left\| \Delta _pf \right\|_{\ell \sp q(\Gamma )} \ge C \| f\|_{D_p(\Gamma )}^{p-1} ,$$ 
where $q$ is the conjugate exponent of $p$.  
\end{enumerate}
\end{corollary}

In particular, if $p=2$, condition (ii) in Corollary \ref{cor:4} can be  regarded as representing a lower estimation of the spectrum $\Delta _2$ with respect to an inner product on $D_2(\Gamma )$.

\vspace{5mm}
\noindent 
Mamoru Tanaka,\\ 
Advanced Institute for Materials Research, Tohoku University, Sendai, 980-8577 Japan\\ 
E-mail: mamoru.tanaka@wpi-aimr.tohoku.ac.jp


\begin{thebibliography}{99}
\bibitem{bibBFGM} 
U. Bader, A. Furman, T. Gelander and N. Monod, 
Property $(T)$ and rigidity for actions on Banach spaces. 
{\it Acta Math.} {\bf 198} (2007), 57--105.  
%
\bibitem{bibBHV}
B. Bekka, P. de la Harpe and A. Valette,  
{\it Kazhdan's Property $(T)$}. 
Cambridge University Press, Cambridge 2008. 
%
\bibitem{bibBL} 
Y. Benyamini and J. Lindestrauss, 
{\it Geometric Nonlinear Functional Analysis, Vol. 1.}  
American Mathematical Society, Providence, RI 2000. 
%
\bibitem{bibBP}
M. Bourdon and H. Pajot, 
Cohomologie $\ell \sp p$ et espaces de Besov.
{\it J. reine angew. Math.} {\bf 558} (2003), 85--108. 
%
\bibitem{bibCDH}
I. Chatterji, C. Dru\c{t}u and F. Haglund,  
Kazhdan and Haagerup properties from the median viewpoint.  
{\it Adv. Math.} {\bf 225} (2010), 882--921. 
%
\bibitem{bibHLP}
G. H. Hardy, J. E. Littlewood and G. P$\acute{\rm o}$lya, 
{\it Inequalities, Second edition}. 
Cambridge University Press, Cambridge 1952. 
%
\bibitem{bibLT1}
J. Lindenstrauss and L. Tzafriri, 
{\it Classical Banach Spaces I, Sequence Spaces}. 
Springer-Verlag, Berlin-New York 1977. 
%
\bibitem{bibLT2}
J. Lindenstrauss and L. Tzafriri, 
{\it Classical Banach Spaces II, Function Spaces}.  
Springer-Verlag, Berlin-New York 1979.  
%
\bibitem{bibM} 
U. F. Mayer,  
Gradient flows on nonpositively curved metric spaces and harmonic maps. 
{\it Comm. Anal. Geom.} {\bf 6} (1998), 199--253. 
%
\bibitem{bibPul03}
M. J. Puls,  
Group cohomology and $L\sp p$-cohomology of finitely generated groups. 
{\it Canad. Math. Bull.} {\bf 46} (2003), 268--276.  
%
\bibitem{bibPul06}
M. J. Puls, 
The first $L\sp p$-cohomology of some finitely generated groups and $p$-harmonic functions.  
{\it J. Funct. Anal.} {\bf 237} (2006), 391--401. 
%
\bibitem{bibT}
M. Tanaka, 
The energy of equivariant maps and a fixed-point property for Busemann nonpositive curvature spaces. 
{\it Trans. Amer. Math. Soc.} {\bf 363} (2011), 1743-1763. 
%
\end{thebibliography}
\end{document}